\documentclass{amsart}
\usepackage{amsmath,amssymb}
\usepackage[dvips]{graphics}
\usepackage[all]{xy}
\newtheorem{Theorem}{Theorem}[section]
\newtheorem{Lemma}[Theorem]{Lemma}

\newtheorem{Corollary}[Theorem]{Corollary}

\theoremstyle{definition}

\theoremstyle{remark}
\newtheorem{Remark}[Theorem]{Remark}

\def\labelenumi{\rm ({\makebox[0.8em][c]{\theenumi}})}
\def\theenumi{\arabic{enumi}}
\def\labelenumii{\rm {\makebox[0.8em][c]{(\theenumii)}}}
\def\theenumii{\roman{enumii}}

\renewcommand{\thefigure}
{\arabic{section}.\arabic{figure}}
\makeatletter
\@addtoreset{figure}{section}
\def\@thmcountersep{-}
\makeatother

\numberwithin{equation}{section}

\begin{document} 

\title[On Conway-Gordon type theorems for graphs in the Petersen family]{On Conway-Gordon type theorems for graphs in the Petersen family}

\author{Hiroka Hashimoto}
\address{Division of Mathematics, Graduate School of Science, Tokyo Woman's Christian University, 2-6-1 Zempukuji, Suginami-ku, Tokyo 167-8585, Japan}
\email{etiscatbird@yahoo.co.jp}

\author{Ryo Nikkuni}
\address{Department of Mathematics, School of Arts and Sciences, Tokyo Woman's Christian University, 2-6-1 Zempukuji, Suginami-ku, Tokyo 167-8585, Japan}
\email{nick@lab.twcu.ac.jp}
\thanks{The second author was partially supported by Grant-in-Aid for Young Scientists (B) (No. 21740046), Japan Society for the Promotion of Science.}


\date{}



\begin{abstract}
For every spatial embedding of each graph in the Petersen family, it is known that the sum of the linking numbers over all of the constituent 2-component links is congruent to 1 modulo 2. In this paper, we give an integral lift of this formula in terms of the square of the linking number and the second coefficient of the Conway polynomial. 
\end{abstract}

\maketitle

\section{Introduction} 

Throughout this paper we work in the piecewise linear category. Let $G$ be a finite graph. An embedding $f$ of $G$ into the $3$-sphere is called a {\it spatial embedding} of $G$ and $f(G)$ is called a {\it spatial graph}. We denote the set of all spatial embeddings of $G$ by ${\rm SE}(G)$. We call a subgraph $\gamma$ of $G$ which is homeomorphic to the circle a {\it cycle} of $G$ and denote the set of all cycles of $G$ by $\Gamma(G)$. In particular, we call a cycle of $G$ which contains exactly $k$ edges a {\it $k$-cycle} of $G$ and denote the set of all $k$-cycles of $G$ by $\Gamma_{k}(G)$. For a positive integer $n$, $\Gamma^{(n)}(G)$ denotes the set of all cycles of $G$ ($=\Gamma(G)$) if $n=1$ and the set of all unions of mutually disjoint $n$ cycles of $G$ if $n\ge 2$. We denote the union of $\Gamma^{(n)}(G)$ over all positive integer $n$ by $\bar{\Gamma}(G)$. For an element $\gamma$ in $\Gamma^{(n)}(G)$ and an element $f$ in ${\rm SE}(G)$, $f(\gamma)$ is none other than a knot in $f(G)$ if $n=1$ and an $n$-component link in $f(G)$ if $n\ge 2$.

A {\it $\triangle Y$-exchange} is an operation to obtain a new graph $G_{Y}$ from a graph $G_{\triangle}$ by removing all edges of a $3$-cycle $\triangle=[uvw]$ of $G_{\triangle}$ with the edges $uv,vw$ and $wu$, and adding a new vertex $x$ and connecting it to each of the vertices $u,v$ and $w$ as illustrated in Fig. \ref{Delta-Y} (we often denote ${ux}\cup {vx}\cup {wx}$ by $Y$). A {\it $Y \triangle$-exchange} is the reverse of this operation. Let $K_{n}$ be the {\it complete graph} on $n$ vertices, namely the simple graph consisting of $n$ vertices in which every pair of distinct vertices is connected by exactly one edge. The set of all graphs obtained from $K_{6}$ by a finite sequence of $\triangle Y$-exchanges and $Y \triangle$-exchanges is called the {\it Petersen family}. The Petersen family consists of seven graphs $K_{6}$, $Q_{7}$, $Q_{8}$, $P_{7}$, $P_{8}$, $P_{9}$ and the {\it Petersen graph} $P_{10}$ as illustrated in Fig. \ref{Petersen}, where an arrow between two graphs indicates the application of a single $\triangle Y$-exchange. For spatial embeddings of a graph in the Petersen family, the following is known. 

\begin{figure}[htbp]
      \begin{center}
\scalebox{0.45}{\includegraphics*{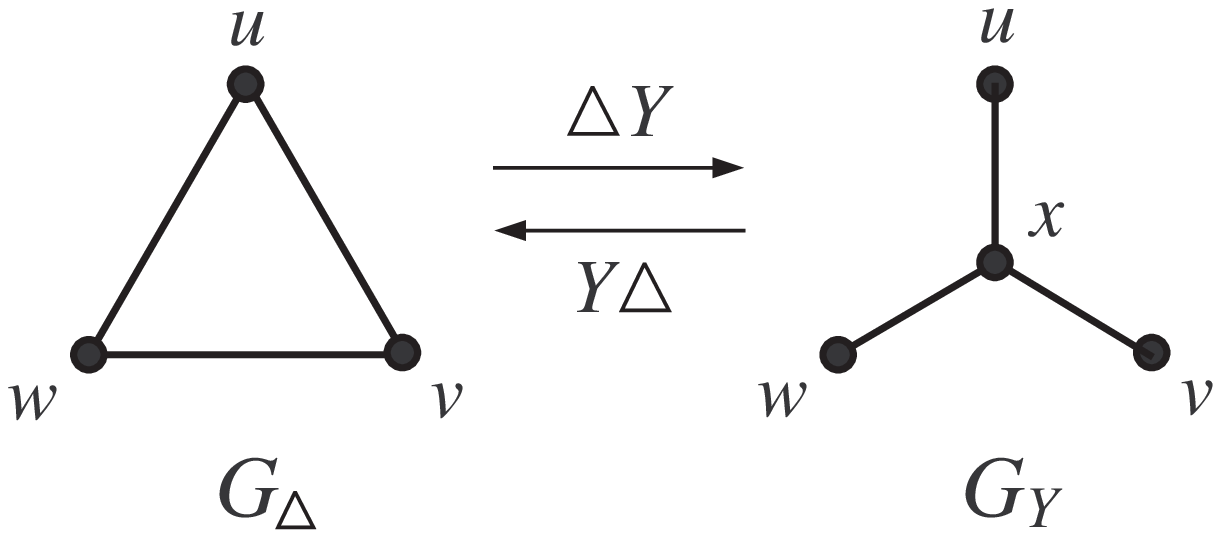}}
      \end{center}
   \caption{}
  \label{Delta-Y}
\end{figure} 
\begin{figure}[htbp]
      \begin{center}
\scalebox{0.4}{\includegraphics*{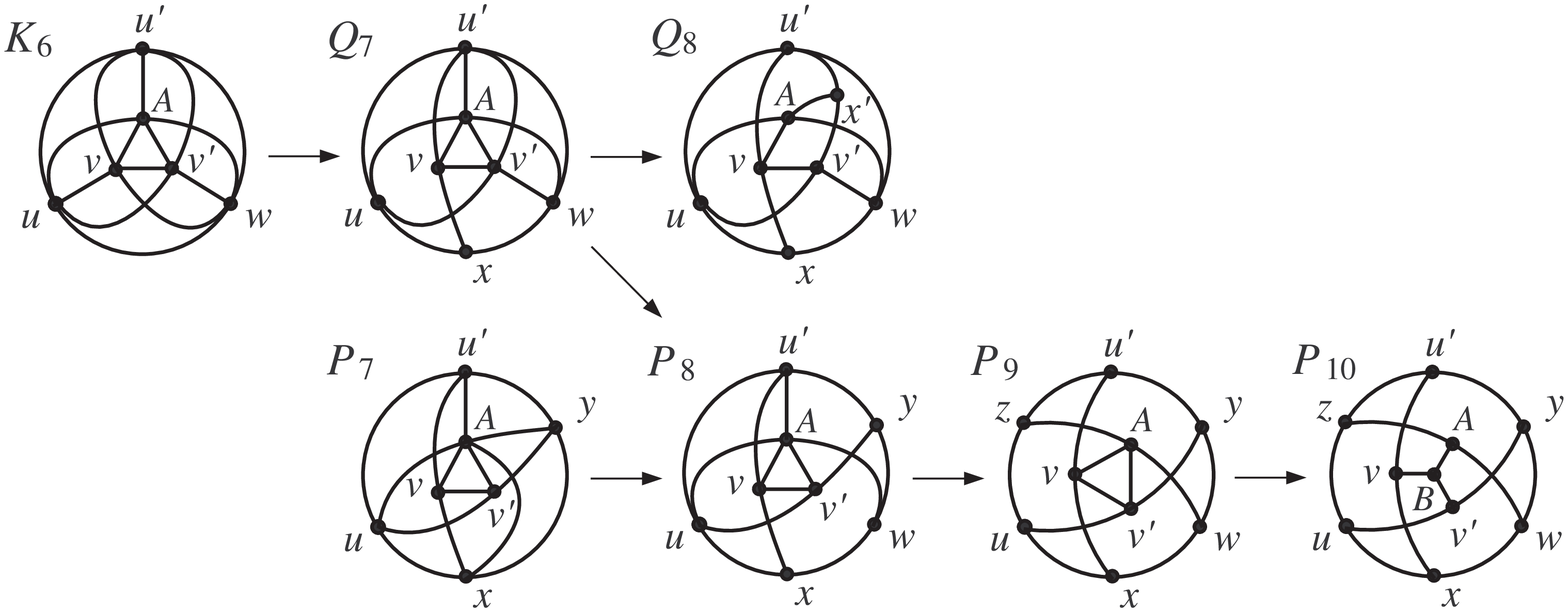}}
      \end{center}
   \caption{}
  \label{Petersen}
\end{figure} 

\begin{Theorem}\label{PF} 
Let $G$ be an element in the Petersen family. For any element $f$ in ${\rm SE}(G)$, it follows that 
\begin{eqnarray*}
\sum_{\gamma\in \Gamma^{(2)}(G)}{\rm lk}(f(\gamma))\equiv 1\pmod{2}, 
\end{eqnarray*}
where ${\rm lk}$ denotes the {\it linking number} in the $3$-sphere. 
\end{Theorem}

We remark here that the case of $G=K_{6}$ in Theorem \ref{PF} is what is called the Conway-Gordon $K_{6}$ theorem \cite{CG83}, and the other cases were shown by Sachs \cite{S84} indirectly, and also pointed out by Taniyama-Yasuhara \cite{TY01}. Theorem \ref{PF} implies that each element $G$ in the Petersen family is {\it intrinsically linked}, that is, for any element $f$ in ${\rm SE}(G)$, there exists an element $\gamma$ in $\Gamma^{(2)}(G)$ such that $f(\gamma)$ is a nonsplittable $2$-component link. It is known that a graph is intrinsically linked if and only if the graph contains an element in the Petersen family as a minor \cite{RST95}. Namely, the Petersen family plays the role of a complete obstruction for graphs not to be intrinsically linked. 

Our purpose in this paper is to give an integral lift of Theorem \ref{PF}. In the following, $a_{i}(L)$ denotes the $i$-th coefficient of the {\it Conway polynomial} for an oriented link $L$.

\begin{Theorem}\label{petersen_refine} 
Let $G$ be an element in the Petersen family. We give the labels for all vertices of $G$ as indicated in Fig. \ref{Petersen}. Let $\omega_{G}$ be a map from $\Gamma(G)$ to the set of integers ${\mathbb Z}$ defined as follows:
\begin{enumerate}
\item If $G=K_{6}$, then for an element $\gamma$ in $\Gamma(K_{6})$, we define 
\begin{eqnarray*}
{\omega_{K_{6}}}(\gamma)=\left\{
       \begin{array}{@{\,}ll}
       1 & \mbox{if $\gamma\in \Gamma_{6}(K_{6})$}\\
       -1 & \mbox{if $\gamma\in \Gamma_{5}(K_{6})$}\\
       0 & \mbox{$otherwise$.}
       \end{array}
     \right.
\end{eqnarray*}

\item If $G=Q_{7}$, then for an element $\gamma$ in $\Gamma(Q_{7})$, we define 
\begin{eqnarray*}
{\omega_{Q_{7}}}(\gamma)=\left\{
       \begin{array}{@{\,}ll}
       1 & \mbox{if $\gamma\in \Gamma_{7}(Q_{7})\cup \Gamma_{6}^{1}(Q_{7})$}\\
       -1 & \mbox{if $\gamma\in \Gamma_{6}^{2}(Q_{7})\cup \Gamma_{5}(Q_{7})$}\\
       0 & \mbox{$otherwise$,}
       \end{array}
     \right.
\end{eqnarray*}
where 
\begin{eqnarray*}
\Gamma_{6}^{1}(Q_{7})&=&\left\{\delta\in \Gamma_{6}(Q_{7})\ |\ \delta \not\ni x\right\},\\
\Gamma_{6}^{2}(Q_{7})&=&\left\{\delta\in \Gamma_{6}(Q_{7})\ |\ \delta\supset \left\{x,u,v,w\right\}\right\}. 
\end{eqnarray*}

\item If $G=Q_{8}$, then for an element $\gamma$ in $\Gamma(Q_{8})$, we define 
\begin{eqnarray*}
{\omega_{Q_{8}}}(\gamma)=\left\{
       \begin{array}{@{\,}ll}
       1 & \mbox{if $\gamma\in \Gamma_{8}(Q_{8})\cup \Gamma_{6}^{1}(Q_{8})$}\\
       -1 & \mbox{if $\gamma\in \Gamma_{6}^{2}(Q_{8})$}\\
       0 & \mbox{$otherwise$,}
       \end{array}
     \right.
\end{eqnarray*}
where 
\begin{eqnarray*}
\Gamma_{6}^{1}(Q_{8})&=&\left\{\delta\in \Gamma_{6}(Q_{8})\ |\ \delta\cap \left\{x,x'\right\}= \emptyset\right\},\\
\Gamma_{6}^{2}(Q_{8})&=&\left\{\delta\in \Gamma_{6}(Q_{8})\ |\ \delta\cap \left\{x,x'\right\}\neq \emptyset\right\}. 
\end{eqnarray*}

\item If $G=P_{7}$, then for an element $\gamma$ in $\Gamma(P_{7})$, we define 
\begin{eqnarray*}
{\omega_{P_{7}}}(\gamma)=\left\{
       \begin{array}{@{\,}ll}
       1 & \mbox{if $\gamma\in \Gamma_{7}(P_{7})$}\\
       -1 & \mbox{if $\gamma\in \Gamma_{5}(P_{7})$}\\
       -2 & \mbox{if $\gamma\in \Gamma_{6}^{1}(P_{7})$}\\
       0 & \mbox{$otherwise$,}
       \end{array}
     \right.
\end{eqnarray*}
where 
\begin{eqnarray*}
\Gamma_{6}^{1}(P_{7})&=&\left\{\delta\in \Gamma_{6}(P_{7})\ |\ \delta\not\ni A\right\}. 
\end{eqnarray*}

\item If $G=P_{8}$, then for an element $\gamma$ in $\Gamma(P_{8})$, we define 
\begin{eqnarray*}
{\omega_{P_{8}}}(\gamma)=\left\{
       \begin{array}{@{\,}ll}
       1 & \mbox{if $\gamma\in \Gamma_{8}(P_{8})\cup \Gamma_{7}^{1}(P_{8})$}\\
       -1 & \mbox{if $\gamma\in \Gamma_{7}^{2}(P_{8})\cup \Gamma_{6}^{1}(P_{8})\cup \Gamma_{5}(P_{8})$}\\
       -2 & \mbox{if $\gamma\in \Gamma_{6}^{2}(P_{8})$}\\
       0 & \mbox{$otherwise$,}
       \end{array}
     \right.
\end{eqnarray*}
where 
\begin{eqnarray*}
\Gamma_{7}^{1}(P_{8})&=&\left\{\delta\in \Gamma_{7}(P_{8})\ |\ \delta \not\supset \left\{x,y,w\right\}\right\},\\
\Gamma_{7}^{2}(P_{8})&=&\left\{\delta\in \Gamma_{7}(P_{8})\ |\ \delta\not\ni A\right\},\\
\Gamma_{6}^{1}(P_{8})&=&\left\{\delta\in \Gamma_{6}(P_{8})\ |\ \delta\ni w\right\},\\
\Gamma_{6}^{2}(P_{8})&=&\left\{\delta\in \Gamma_{6}(P_{8})\ |\ \delta\cap\{A,w\}=\emptyset\right\}. 
\end{eqnarray*}

\item If $G=P_{9}$, then for an element $\gamma$ in $\Gamma(P_{9})$, we define 
\begin{eqnarray*}
{\omega_{P_{9}}}(\gamma)=\left\{
       \begin{array}{@{\,}ll}
       1 & \mbox{if $\gamma\in \Gamma_{9}(P_{9})\cup \Gamma_{8}^{1}(P_{9})$}\\
       -1 & \mbox{if $\gamma\in \Gamma_{7}^{1}(P_{9})\cup \Gamma_{6}^{1}(P_{9})\cup \Gamma_{5}(P_{9})$}\\
       -2 & \mbox{if $\gamma\in \Gamma_{6}^{2}(P_{9})$}\\
       0 & \mbox{$otherwise$,}
       \end{array}
     \right.
\end{eqnarray*}
where 
\begin{eqnarray*}
\Gamma_{8}^{1}(P_{9})&=&\left\{\delta\in \Gamma_{8}(P_{9})\ |\ \delta \supset \left\{A,v,v'\right\}\right\},\\
\Gamma_{7}^{1}(P_{9})&=&\left\{\delta\in \Gamma_{7}(P_{9})\ |\ \delta\not\supset \left\{A,v,v'\right\}\right\},\\
\Gamma_{6}^{1}(P_{9})&=&\left\{\delta\in \Gamma_{6}(P_{9})\ |\ \delta \supset \left\{A,v,v'\right\}\right\},\\
\Gamma_{6}^{2}(P_{9})&=&\left\{\delta\in \Gamma_{6}(P_{9})\ |\ \delta\not\supset \left\{A,v,v'\right\}\right\}. 
\end{eqnarray*}

\item If $G=P_{10}$, then for an element $\gamma$ in $\Gamma(P_{10})$, we define 
\begin{eqnarray*}
{\omega_{P_{10}}}(\gamma)=\left\{
       \begin{array}{@{\,}ll}
       1 & \mbox{if $\gamma\in \Gamma_{9}(P_{10})$}\\
       -2 & \mbox{if $\gamma\in \Gamma_{6}(P_{10})$}\\
       -1 & \mbox{if $\gamma\in \Gamma_{5}(P_{10})$}\\
       0 & \mbox{$otherwise$.}
       \end{array}
     \right.
\end{eqnarray*}
\end{enumerate}
Then for any element $f$ in ${\rm SE}(G)$, it follows that 
\begin{eqnarray*}
2\sum_{\gamma\in \Gamma(G)}\omega_{G}(\gamma)a_{2}(f(\gamma))
=
\left(
\sum_{\gamma\in \Gamma^{(2)}(G)}{\rm lk}(f(\gamma))^{2}
\right)
-1. 
\end{eqnarray*}
\end{Theorem}

Note that Theorem \ref{PF} can be obtained from Theorem \ref{petersen_refine} by taking the modulo two reduction. We also should remark here that Theorem \ref{petersen_refine} was already known in the case of $G=K_{6}$ (Nikkuni \cite{N09b}), $P_{7}$ (O'Donnol \cite{D10}) and $Q_{7}$ (Nikkuni-Taniyama \cite{NT12}). The other cases are new.  

We say that an element $f$ in ${\rm SE}(G)$ is {\it knotted} if $f(G)$ contains a nontrivial knot, and {\it complexly algebraically linked} if $f(G)$ contains a $2$-component link whose linking number is not equal to $0,\pm 1$ or a pair of $2$-component links with nonzero linking number \cite{D10}. Then O'Donnol showed the following. 

\begin{Theorem}\label{CAlinked} 
{\rm (O'Donnol \cite{D10})} Let $G$ be an element in the Petersen family and $f$ an element in ${\rm SE}(G)$. If $f$ is complexly algebraically linked, then $f$ is knotted. 
\end{Theorem}

In \cite{D10}, the place of the cycle whose image is a nontrivial knot was not examined. As an application of Theorem \ref{petersen_refine}, we give an alternative proof of Theorem \ref{CAlinked} in more refined form as follows. 

\begin{Corollary}\label{CAlinked_refine} 
Let $G$ be an element in the Petersen family and $\omega_{G}$ the map from $\Gamma(G)$ to ${\mathbb Z}$ in Theorem \ref{petersen_refine}. Let $f$ be an element in ${\rm SE}(G)$. If $f$ is complexly algebraically linked, then it follows that 
\begin{eqnarray*}
\sum_{\gamma\in \Gamma(G)}{\omega_{G}}(\gamma)a_{2}(f(\gamma))\ge 1.
\end{eqnarray*}
\end{Corollary}

\begin{proof}
By Theorem \ref{PF}, it follows that $f$ is complexly algebraically linked if and only if 
\begin{eqnarray}\label{CAlk}
\sum_{\gamma\in \Gamma^{(2)}(G)}{\rm lk}(f(\gamma))^{2}\ge 3.
\end{eqnarray}
Then by Theorem \ref{petersen_refine} and (\ref{CAlk}), it follows that 
\begin{eqnarray*}
\sum_{\gamma\in \Gamma(G)}{\omega_{G}}(\gamma)a_{2}(f(\gamma))
=\frac{1}{2}\left\{
\left(
\sum_{\gamma\in \Gamma^{(2)}(G)}{\rm lk}(f(\gamma))^{2}
\right)
-1
\right\}
\ge 1. 
\end{eqnarray*}
Thus we have the desired conclusion. 
\end{proof}

By Corollary \ref{CAlinked_refine}, there exists an element $\gamma_{0}$ in $\Gamma(G)$ with $\omega_{G}(\gamma_{0})\neq 0$ such that $a_{2}(f(\gamma_{0}))\neq 0$. Namely, Corollary \ref{CAlinked_refine} refines Theorem \ref{CAlinked} by identifying the cycles that might be nontrivial knots in $f(G)$. 

In the next section, we prepare general results for $\triangle Y$-exchanges and the Conway-Gordon type theorems which are based on \cite{NT12}. We give a proof of Theorem \ref{petersen_refine} in section $3$.

\section{General results} 

Let $G_{\triangle}$ and $G_{Y}$ be two graphs such that $G_{Y}$ is obtained from $G_{\triangle}$ by a single $\triangle Y$-exchange. We denote the set of all elements in $\bar{\Gamma}(G_{\triangle})$ containing $\triangle$ by $\bar{\Gamma}_{\triangle}(G_{\triangle})$. 
Let $\gamma'$ be an element in $\bar{\Gamma}(G_{\triangle})$ which does not contain $\triangle$. Then there exists an element $\bar{\Phi}(\gamma')$ in $\bar{\Gamma}(G_{Y})$ such that $\gamma'\setminus \triangle={{\bar{\Phi}}(\gamma')}\setminus Y$. It is easy to see that the correspondence from $\gamma'$ to $\bar{\Phi}(\gamma')$ defines a surjective map 
\begin{eqnarray}\label{phi}
\bar{\Phi}=\bar{\Phi}_{G_{\triangle},G_{Y}}:\bar{\Gamma}(G_{\triangle})\setminus \bar{\Gamma}_{\triangle}(G_{\triangle})\longrightarrow \bar{\Gamma}(G_{Y}).
\end{eqnarray}
In particular, if $\gamma'$ is an element in $\Gamma^{(n)}(G_{\triangle})\setminus \bar{\Gamma}_{\triangle}(G_{\triangle})$ then $\bar{\Phi}(\gamma')$ is an element in $\Gamma^{(n)}(G_{Y})$. This implies that the restriction map of $\bar{\Phi}$ on $\Gamma^{(n)}(G_{\triangle})\setminus \bar{\Gamma}_{\triangle}(G_{\triangle})$ induces a surjective map from $\Gamma^{(n)}(G_{\triangle})\setminus \bar{\Gamma}_{\triangle}(G_{\triangle})$ to $\Gamma^{(n)}(G_{Y})$, and thus it is clear that if $\Gamma^{(n)}(G_{\triangle})$ is an empty set, then $\Gamma^{(n)}(G_{Y})$ is also an empty set. The inverse image of an element $\gamma$ in $\bar{\Gamma}(G_{Y})$ by $\bar{\Phi}$ contains at most two elements in $\bar{\Gamma}(G_{\triangle})\setminus \bar{\Gamma}_{\triangle}(G_{\triangle})$. Fig. \ref{not_inj2} illustrates the case that the inverse image of $\gamma$ by $\bar{\Phi}$ consists of  exactly two elements. In general, the inverse image of $\gamma$ by $\bar{\Phi}$ consists of exactly one element if and only if $\gamma$ contains $u,v,w$ and $x$, or $\gamma$ does not contain $x$. 

\begin{figure}[htbp]
      \begin{center}
\scalebox{0.425}{\includegraphics*{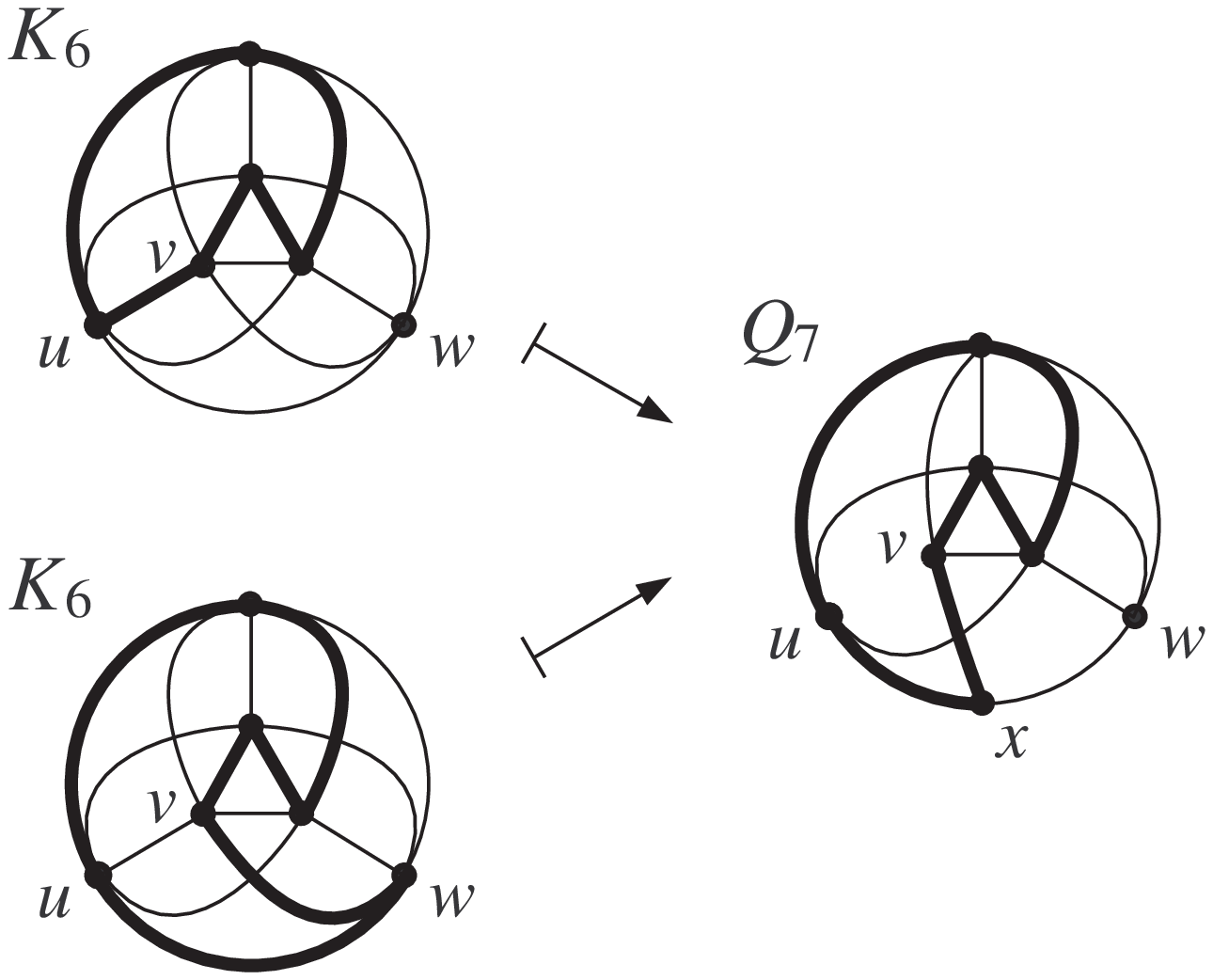}}
      \end{center}
   \caption{}
  \label{not_inj2}
\end{figure} 

Let $A$ be an additive group. We say that an $A$-valued unoriented link invariant $\alpha$ is {\it compressible} if $\alpha(L)=0$ for any unoriented link $L$ which has a component $K$ bounding a disk $D$ in the $3$-sphere with $D\cap L = \partial D = K$. Suppose that for each element $\gamma'$ in $\bar{\Gamma}(G_{\triangle})$, an $A$-valued unoriented link invariant $\alpha_{\gamma'}$ is assigned. Then for each element $\gamma$ in $\bar{\Gamma}(G_{Y})$, we define an $A$-valued unoriented link invariant $\tilde{\alpha}_\gamma$ by
\begin{eqnarray*}
\tilde{\alpha}_\gamma(L)=\sum_{\gamma'\in\bar{\Phi}^{-1}(\gamma)}\alpha_{\gamma'}(L)
\end{eqnarray*}
for an unoriented link $L$. Then the following theorem holds. 

\begin{Theorem}\label{main2} 
{\rm (Nikkuni-Taniyama \cite{NT12})} 
Suppose that $\alpha_{\gamma'}$ is compressible for each element $\gamma'$ in $\bar{\Gamma}_\triangle(G_\triangle)$. Suppose that there exists a fixed element $c$ in $A$ such that
\begin{eqnarray*}
\sum_{\gamma'\in\bar{\Gamma}(G_{\triangle})}\alpha_{\gamma'}(g(\gamma'))=c
\end{eqnarray*}
for any element $g$ in ${\rm SE}(G_{\triangle})$. Then we have
\begin{eqnarray*}
\sum_{\gamma\in\bar{\Gamma}(G_{Y})}\tilde{\alpha}_\gamma(f(\gamma))=c
\end{eqnarray*}
for any element $f$ in ${\rm SE}(G_{Y})$.
\end{Theorem}

By an application of Theorem \ref{main2}, the following is shown. 

\begin{Theorem}\label{main3} 
{\rm (Nikkuni-Taniyama \cite{NT12})} 
Let $G$ be an element in the Petersen family. Then, there exists a map $\omega_{G}$ from $\Gamma(G)$ to ${\mathbb Z}$ such that for any element $f$ in ${\rm SE}(G)$, it follows that 
\begin{eqnarray*}
2\sum_{\gamma\in \Gamma(G)}\omega_{G}(\gamma)a_{2}(f(\gamma))
=
\left(
\sum_{\gamma\in \Gamma^{(2)}(G)}{\rm lk}(f(\gamma))^{2}
\right)
-1.
\end{eqnarray*}
\end{Theorem}

We give a proof of Theorem \ref{main3} for the reader's convenience. 

\begin{proof}[Proof of Theorem \ref{main3}]
Let $G_{\triangle}$ and $G_{Y}$ be two elements in the Petersen family such that $G_{Y}$ is obtained from $G_{\triangle}$ by a single $\triangle Y$-exchange. Assume that there exists a map $\omega$ from $\bar{\Gamma}(G_{\triangle})$ to ${\mathbb Z}$ such that for any element $g$ in ${\rm SE}(G_{\triangle})$, it follows that  
\begin{eqnarray}\label{assump}
2\sum_{\gamma'\in \Gamma(G_{\triangle})}\omega(\gamma')a_{2}(g(\gamma'))
=
\left(
\sum_{\gamma'\in \Gamma^{(2)}(G_{\triangle})}{\rm lk}(g(\gamma'))^{2}
\right)
-1.
\end{eqnarray}
For each element $\gamma'$ in $\bar{\Gamma}(G_{\triangle})$, we define an integer-valued unoriented link invariant $\alpha_{\gamma'}$ of an unoriented link $L$ as follows. Note that $G_{\triangle}$ is obtained from $K_{6}$ or $P_{7}$ by a finite sequence of $\triangle Y$-exchanges. Since both $\Gamma^{(n)}(K_{6})$ and $\Gamma^{(n)}(P_{7})$ are empty sets for $n\ge 3$, we have $\Gamma^{(n)}(G_{\triangle})$ is an empty set for $n\ge 3$. If $\gamma'$ is an element in $\Gamma(G_{\triangle})$, then $\alpha_{\gamma'}(L)=2\omega(\gamma')a_{2}(L)$ if $L$ is a knot and $0$ if $L$ is not a knot. If $\gamma'$ is an element in $\Gamma^{(2)}(G_{\triangle})$, then $\alpha_{\gamma'}(L)=-{\rm lk}(L)^{2}$ if $L$ is a $2$-component link and $0$ if $L$ is not a $2$-component link. Then by (\ref{assump}), we have 
\begin{eqnarray}\label{const_k6}
\sum_{\gamma'\in \bar{\Gamma}(G_{\triangle})}\alpha_{\gamma'}(g(\gamma'))=-1. 
\end{eqnarray}
Note that $\alpha_{\gamma'}$ is compressible for any element $\gamma'$ in $\bar{\Gamma}(G_{\triangle})$. Thus by Theorem \ref{main2} and (\ref{const_k6}), for any element $f$ in ${\rm SE}(G_{Y})$, it follows that  
\begin{eqnarray}\label{const_q7}
\sum_{\gamma\in \bar{\Gamma}(G_{Y})}\tilde{\alpha}_{\gamma}(f(\gamma))=-1.
\end{eqnarray}
Now we define a map $\tilde{\omega}$ from ${\Gamma}(G_{Y})$ to ${\mathbb Z}$ by 
\begin{eqnarray}\label{c}
\tilde{\omega}(\gamma)
=
\sum_{\gamma'\in \bar{\Phi}^{-1}(\gamma)}\omega(\gamma')
\end{eqnarray}
for an element $\gamma$ in ${\Gamma}(G_{Y})$. Let $\gamma$ be an element in $\bar{\Gamma}(G_{Y})$. Note that $\Gamma^{(n)}(G_{Y})$ is also an empty set for $n\ge 3$. If $\gamma$ belongs to $\Gamma(G_{Y})$, then by (\ref{c}), we have 
\begin{eqnarray}\label{omega}
\tilde{\alpha}_{\gamma}(f(\gamma))
=
2\sum_{\gamma'\in {\Phi}^{-1}(\gamma)}\omega(\gamma')a_{2}(f(\gamma))
=
2\tilde{\omega}(\gamma)a_{2}(f(\gamma)). 
\end{eqnarray}
If $\gamma$ belongs to $\Gamma^{(2)}(G_{Y})$, then $\bar{\Phi}^{-1}(\gamma)$ consists of exactly one element because each union of two mutually disjoint cycles of a graph in the Petersen family contains all of the vertices of the graph. Then we have 
\begin{eqnarray}\label{xi}
\tilde{\alpha}_{\gamma}(f(\gamma))
=
\alpha_{\bar{\Phi}^{-1}(\gamma)}(f(\gamma))
=
-{\rm lk}(f(\gamma))^{2}. 
\end{eqnarray}
Thus by combining (\ref{const_q7}), (\ref{omega}) and (\ref{xi}), we have 
\begin{eqnarray}\label{a}
2\sum_{\gamma\in \Gamma(G_{Y})}\tilde{\omega}(\gamma)a_{2}(f(\gamma))
-\sum_{\gamma\in \Gamma^{(2)}(G_{Y})}{\rm lk}(f(\gamma))^{2}=-1. 
\end{eqnarray}
Namely, for any element $f$ in ${\rm SE}(G_{Y})$, it follows that 
\begin{eqnarray*}
2\sum_{\gamma\in \Gamma(G_{Y})}\tilde{\omega}(\gamma)a_{2}(f(\gamma))
=
\left(
\sum_{\gamma\in \Gamma^{(2)}(G_{Y})}{\rm lk}(f(\gamma))^{2}
\right)
-1.
\end{eqnarray*}
As we remarked before, the cases of $K_{6}$ and $P_{7}$ have already shown by \cite{N09b} and \cite{D10}, respectively. Thus by repeating the argument as above, we have the desired conclusion. 
\end{proof}

In the following, we show two lemmas which are useful in proving Theorem \ref{petersen_refine}. We say that two cycles of a graph are {\it edge-disjoint} if the intersection of them does not contain an edge. Let $G$ be a graph and $\triangle_{1},\triangle_{2},\ldots,\triangle_{k}$ $3$-cycles of $G$ such that $\triangle_{i} \cap \triangle_{j}$ is edge-disjoint for $i\neq j$. Then we also can regard $\triangle_{i}$ as a $3$-cycle of the graph obtained from $G$ by a finite sequence of $\triangle Y$-exchanges at $\triangle_{j}$'s for $i\neq j$. Let $G_{l}$ be a graph obtained from $G_{l-1}$ by a single $\triangle Y$-exchange at $\triangle_{l}$, where $G=G_{0}$ ($l=1,2,\ldots,k$). On the other hand, let $\sigma$ be a permutation of order $k$. Let $G'_{l}$ be a graph obtained from $G'_{l-1}$ by a single $\triangle Y$-exchange at $\triangle_{\sigma(l)}$, where $G=G'_{0}$ ($l=1,2,\ldots,k$). Note that $G_{k}=G'_{k}$. 

\begin{Lemma}\label{inverse} 
For any element $\gamma$ in $\bar{\Gamma}(G_{k})$, it follows that 
\begin{eqnarray*}
&&(\bar{\Phi}_{G_{k-1},G_{k}}
\circ \bar{\Phi}_{G_{k-2},G_{k-1}}
\circ \cdots 
\circ \bar{\Phi}_{G_{0},G_{1}}
)^{-1}(\gamma)\\
&=&
(\bar{\Phi}_{G'_{k-1},G'_{k}}
\circ \bar{\Phi}_{G'_{k-2},G'_{k-1}}
\circ \cdots 
\circ \bar{\Phi}_{G'_{0},G'_{1}})^{-1}
(\gamma).
\end{eqnarray*}
\end{Lemma}

\begin{proof}
Let $\gamma'$ be an element in the inverse image of $\gamma$ by $\bar{\Phi}_{G_{k-1},G_{k}}
\circ \bar{\Phi}_{G_{k-2},G_{k-1}}
\circ \cdots 
\circ \bar{\Phi}_{G_{0},G_{1}}$. Since $\triangle_{i}\cap \triangle_{j}$ is edge-disjoint for $i\neq j$, we have 
\begin{eqnarray*}
\gamma
&=&\bar{\Phi}_{G_{k-1},G_{k}}
\circ \bar{\Phi}_{G_{k-2},G_{k-1}}
\circ \cdots 
\circ \bar{\Phi}_{G_{0},G_{1}}(\gamma')\\
&=& \bar{\Phi}_{G'_{k-1},G'_{k}}
\circ \bar{\Phi}_{G'_{k-2},G'_{k-1}}
\circ \cdots 
\circ \bar{\Phi}_{G'_{0},G'_{1}}(\gamma'). 
\end{eqnarray*}
Thus $\gamma'$ also be an element in the inverse image of $\gamma$ by 
$\bar{\Phi}_{G'_{k-1},G'_{k}}
\circ \bar{\Phi}_{G'_{k-2},G'_{k-1}}
\circ \cdots 
\circ \bar{\Phi}_{G'_{0},G'_{1}}$. 
This implies the result. 
\end{proof}

We assign an $A$-valued unoriented link invariant $\alpha_{\gamma'}$ for each element $\gamma'$ in $\bar{\Gamma}(G)$. Then for an element $\gamma$ in $\bar{\Gamma}(G_{l})$, we define an $A$-valued unoriented link invariant $\alpha_{\gamma}^{(l)}$ by $\alpha_{\gamma}^{(l)}=\alpha_{\gamma}$ if $l=0$ and $\alpha_{\gamma}^{(l)}=\widetilde{\alpha^{(l-1)}}_{\gamma}$ if $l=1,2,\ldots,k$ with respect to the sequence of $\triangle Y$-exchanges at $\triangle_{1},\triangle_{2},\ldots,\triangle_{k}$. On the other hand, we define an $A$-valued unoriented link invariant $\beta_{\gamma}^{(l)}$ by $\beta_{\gamma}^{(l)}=\alpha_{\gamma}$ if $l=0$ and $\beta_{\gamma}^{(l)}=\widetilde{\beta^{(l-1)}}_{\gamma}$ if $l=1,2,\ldots,k$ with respect to the sequence of $\triangle Y$-exchanges at $\triangle_{\sigma(1)},\triangle_{\sigma(2)},\ldots,\triangle_{\sigma(k)}$. Then we have the following. 

\begin{Lemma}\label{ab} 
For any element $\gamma$ in $\bar{\Gamma}(G_{k})$, it follows that $\alpha_{\gamma}^{(k)}=\beta_{\gamma}^{(k)}$. 
\end{Lemma}

\begin{proof}
Let $\gamma$ be an element in $\bar{\Gamma}(G_{k})$. Then for an unoriented link $L$, by Lemma \ref{inverse}, we have 
\begin{eqnarray*}
\alpha_{\gamma}^{(k)}(L)
&=& \sum_{\gamma'\in (\bar{\Phi}_{G_{k-1},G_{k}}
\circ \bar{\Phi}_{G_{k-2},G_{k-1}}
\circ \cdots 
\circ \bar{\Phi}_{G_{0},G_{1}})^{-1}(\gamma)}
\alpha_{\gamma'}(L)\\
&=& \sum_{\gamma'\in (\bar{\Phi}_{G'_{k-1},G'_{k}}
\circ \bar{\Phi}_{G'_{k-2},G'_{k-1}}
\circ \cdots 
\circ \bar{\Phi}_{G'_{0},G'_{1}})^{-1}(\gamma)}
\alpha_{\gamma'}(L)\\
&=& \beta_{\gamma}^{(k)}(L). 
\end{eqnarray*}
Thus we have the result. 
\end{proof}

\section{Proof of Theorem \ref{petersen_refine}} 

Let $G$ be a graph and $T=\{\triangle_{1},\triangle_{2},\ldots,\triangle_{k}\}$ a set of mutually edge-disjoint $3$-cycles of $G$. We say that $T$ is {\it stable} if for any $l$-element subset $\{\triangle_{i_{1}},\triangle_{i_{2}},\ldots,\triangle_{i_{l}}\}$ of $T$ ($1\le l< k$), $\triangle Y$-exchanges at $\triangle_{i_{1}},\triangle_{i_{2}},\ldots,\triangle_{i_{l}}$ produce the same graph up to isomorphism. 

\begin{proof}[Proof of Theorem \ref{petersen_refine}] 
We denote the $3$-cycles $[uvw]$, $[u'v'w]$, $[uu'A]$, $[vv'A]$ and $[u'v'A]$ of $K_{6}$ by $\triangle_{1}$, $\triangle_{2}$, $\triangle_{3}$, $\triangle_{4}$ and $\triangle_{5}$, respectively. Note that $\triangle_{i}\cap \triangle_{j}$ is edge-disjoint for $1\le i<j\le 4$, and $\triangle_{1}\cap \triangle_{5}$ is also edge-disjoint. 

(1) Let $G$ be $K_{6}$. Then this case has been shown in \cite{N09b}. 

(2) Let $G$ be $Q_{7}$ which is obtained from $K_{6}$ by a single $\triangle Y$-exchange at $\triangle_{1}$. Though this case has been shown in \cite{NT12}, we give it again for the reader's convenience. Let $\omega_{K_{6}}$ be the map from ${\Gamma}(K_{6})$ to ${\mathbb Z}$ as in (1) and $\tilde{\omega}_{K_{6}}$ the map from ${\Gamma}(Q_{7})$ to ${\mathbb Z}$ defined by (\ref{c}) with respect to the $\triangle Y$-exchange at $\triangle_{1}$. In the following we show $\tilde{\omega}_{K_{6}}=\omega_{Q_{7}}$. Then the result follows from the proof of Theorem \ref{main3}. Let $\gamma$ be an element in $\Gamma(Q_{7})$. If $\gamma$ belongs to $\Gamma_{7}(Q_{7})$, then there uniquely exists an element $\gamma'$ in $\Gamma_{6}(K_{6})$ such that $\bar{\Phi}^{-1}_{K_{6},Q_{7}}(\gamma)=\{\gamma'\}$. Thus we have $\tilde{\omega}_{K_{6}}(\gamma)=\omega_{K_{6}}(\gamma')=1$. If $\gamma$ belongs to $\Gamma_{6}(Q_{7})$, then $\gamma$ does not contain exactly one of the vertices of $Q_{7}$. Then it is sufficient to consider the following three cases up to symmetry of $Q_{7}$: (i) If $\gamma$ does not contain $x$, then there uniquely exists an element $\gamma'$ in $\Gamma_{6}(K_{6})$ such that $\bar{\Phi}^{-1}_{K_{6},Q_{7}}(\gamma)=\{\gamma'\}$. Thus we have $\tilde{\omega}_{K_{6}}(\gamma)=\omega_{K_{6}}(\gamma')=1$. (ii) If $\gamma$ does not contain $u$, then there exists an element $\gamma'_{1}$ in $\Gamma_{5}(K_{6})$ and an element $\gamma'_{2}$ in $\Gamma_{6}(K_{6})$ such that $\bar{\Phi}^{-1}_{K_{6},Q_{7}}(\gamma)=\{\gamma'_{1},\gamma'_{2}\}$. Thus we have $\tilde{\omega}_{K_{6}}(\gamma)=\omega_{K_{6}}(\gamma'_{1})+\omega_{K_{6}}(\gamma'_{2})=-1+1=0$. (iii) If $\gamma$ does not contain $u'$, then there uniquely exists an element $\gamma'$ in $\Gamma_{5}(K_{6})$ such that $\bar{\Phi}^{-1}_{K_{6},Q_{7}}(\gamma)=\{\gamma'\}$. Thus we have $\tilde{\omega}_{K_{6}}(\gamma)=\omega_{K_{6}}(\gamma')=-1$. If $\gamma$ belongs to $\Gamma_{5}(Q_{7})$, then the inverse image of $\gamma$ by $\bar{\Phi}$ consists of exactly one element in $\Gamma_{5}(K_{6})$ or a pair of an element in $\Gamma_{5}(K_{6})$ and an element in $\Gamma_{4}(K_{6})$. Thus in any case we have $\tilde{\omega}_{K_{6}}(\gamma)=-1$. 
If $\gamma$ belongs to $\Gamma(Q_{7})\setminus \cup_{k=5}^{7}\Gamma_{k}(Q_{7})$, we have $\tilde{\omega}_{K_{6}}(\gamma)=0$. In conclusion, we see that 
\begin{eqnarray*}
\tilde{\omega}_{K_{6}}(\gamma)=\left\{
       \begin{array}{@{\,}ll}
       1 & \mbox{if $\gamma\in \Gamma_{7}(Q_{7})\cup \left\{\delta\in \Gamma_{6}(Q_{7})\ |\ \delta\not\ni x\right\}$}\\
       -1 & \mbox{if $\gamma\in \left\{\delta\in \Gamma_{6}(Q_{7})\ |\ \delta\ni x,u,v,w\right\}\cup \Gamma_{5}(Q_{7})$}\\
       0 & \mbox{otherwise}
       \end{array}
     \right.
\end{eqnarray*}
for an element $\gamma$ in $\Gamma(Q_{7})$. Thus it follows that $\tilde{\omega}_{K_{6}}=\omega_{Q_{7}}$.

(3) Let $G$ be $Q_{8}$ which is obtained from $Q_{7}$ by a single $\triangle Y$-exchange at $\triangle_{5}$. Let $\omega_{Q_{7}}$ be the map from ${\Gamma}(Q_{7})$ to ${\mathbb Z}$ as in (2) and $\tilde{\omega}_{Q_{7}}$ the map from ${\Gamma}(Q_{8})$ to ${\mathbb Z}$ defined by (\ref{c}) with respect to the $\triangle Y$-exchange at $\triangle_{5}$. In the following we show $\tilde{\omega}_{Q_{7}}=\omega_{Q_{8}}$. Let $\gamma$ be an element in $\Gamma(Q_{8})$. Note that $\Gamma_{k}(Q_{8})=\emptyset$ if $k\neq 4,6,8$. 
If $\gamma$ belongs to $\Gamma_{8}(Q_{8})$, then there uniquely exists an element $\gamma'$ in $\Gamma_{7}(Q_{7})$ such that $\bar{\Phi}^{-1}_{Q_{7},Q_{8}}(\gamma)=\{\gamma'\}$. Thus we have $\tilde{\omega}_{Q_{7}}(\gamma)=\omega_{Q_{7}}(\gamma')=1$. 
If $\gamma$ belongs to $\Gamma_{6}(Q_{8})$, then $\gamma$ does not contain exactly two of the vertices of $Q_{8}$. Then it is sufficient to consider the following three cases up to symmetry of $Q_{8}$: (i) If $\gamma$ does not contain either $x$ or $x'$, then there uniquely exists an element $\gamma'$ in $\Gamma_{6}(Q_{7})$ such that $\bar{\Phi}^{-1}_{Q_{7},Q_{8}}(\gamma)=\{\gamma'\}$. Since $\gamma'$ does not contain $x$, we have $\tilde{\omega}_{Q_{7}}(\gamma)=\omega_{Q_{7}}(\gamma')=1$. 
(ii) If $\gamma$ contains exactly one of $x$ and $x'$, then we may assume that $\gamma$ contains $x'$ by the following reason. Note that $\{\triangle_{1},\triangle_{5}\}$ is stable. Actually, the graph $Q'_{7}$ which is obtained from $K_{6}$ by a single $\triangle Y$-exchange at $\triangle_{5}$ is isomorphic to $Q_{7}$, and $Q_{8}$ is also obtained from $Q'_{7}$ by a single $\triangle Y$-exchange at $\triangle_{1}$, see Fig. \ref{Q8iso}. Then the map from $\Gamma(Q_{8})$ to ${\mathbb Z}$ obtained from $\omega_{K_{6}}$ by $\bar{\Phi}_{Q'_{7},Q_{8}}\circ\bar{\Phi}_{K_{6},Q'_{7}}$ coincides with $\tilde{\omega}_{Q_{7}}$ by (\ref{omega}) and Lemma \ref{ab}. Thus the case that $\gamma$ contains $x$ and the case that $\gamma$ contains $x'$ are compatible through the isomorphism between $Q'_{7}$ and $Q_{7}$. From now on, to simplify the description of the proof we often use the argument of such a compatibility without any notice. Now assume that $\gamma$ contains $x'$. Then there uniquely exists an element $\gamma'$ in $\Gamma_{5}(Q_{7})$ such that $\bar{\Phi}^{-1}_{Q_{7},Q_{8}}(\gamma)=\{\gamma'\}$. Thus we have $\tilde{\omega}_{Q_{7}}(\gamma)=\omega_{Q_{7}}(\gamma')=-1$. 
(iii) If $\gamma$ contains both $x$ and $x'$, then we may assume that $\gamma$ does not contain either $u$ or $u'$. Then there exists an element $\gamma'_{1}$ in $\Gamma_{5}(Q_{7})$ and an element $\gamma'_{2}$ in $\Gamma_{6}(Q_{7})$ such that $\bar{\Phi}^{-1}_{Q_{7},Q_{8}}(\gamma)=\{\gamma'_{1},\gamma'_{2}\}$. Since $\gamma'_{2}$ does not contain $u$, we have $\tilde{\omega}_{Q_{7}}(\gamma)=\omega_{Q_{7}}(\gamma'_{1})+\omega_{Q_{7}}(\gamma'_{2})=-1+0=-1$. 
If $\gamma$ belongs to $\Gamma_{4}(Q_{8})$, we have $\tilde{\omega}_{Q_{7}}(\gamma)=0$. In conclusion, we see that 
\begin{eqnarray*}
\tilde{\omega}_{Q_{7}}(\gamma)=\left\{
       \begin{array}{@{\,}ll}
       1 & \mbox{if $\gamma\in \Gamma_{8}(Q_{8})\cup \left\{\delta\in \Gamma_{6}(Q_{8})\ |\ \delta\cap \{x,x'\}= \emptyset\right\}$}\\
       -1 & \mbox{if $\gamma\in \left\{\delta\in \Gamma_{6}(Q_{8})\ |\ \delta\cap \{x,x'\}\neq \emptyset\right\}$}\\
       0 & \mbox{otherwise}
       \end{array}
     \right.
\end{eqnarray*}
for an element $\gamma$ in $\Gamma(Q_{8})$. Thus it follows that $\tilde{\omega}_{Q_{7}}=\omega_{Q_{8}}$.

\begin{figure}[htbp]
      \begin{center}
\scalebox{0.4}{\includegraphics*{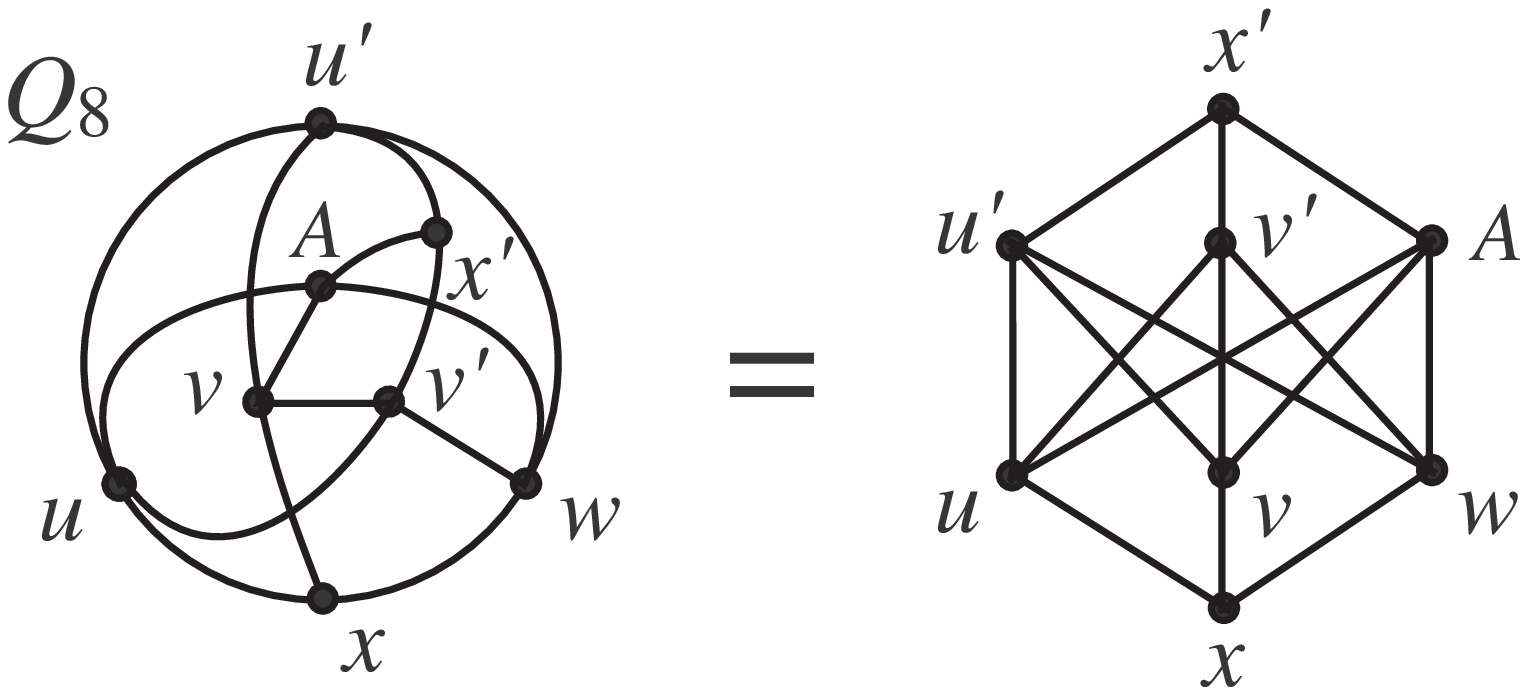}}
      \end{center}
   \caption{}
  \label{Q8iso}
\end{figure} 

(4) Let $G$ be $P_{7}$. Then this case has been shown in \cite{D10}.

(5) Let $G$ be $P_{8}$ which is obtained from $Q_{7}$ by a single $\triangle Y$-exchange at $\triangle_{2}$. Note that $\{\triangle_{1},\triangle_{2}\}$ is stable, see Fig. \ref{P8iso}. Let $\omega_{Q_{7}}$ be the map from ${\Gamma}(Q_{7})$ to ${\mathbb Z}$ as in (2) and $\tilde{\omega}_{Q_{7}}$ the map from ${\Gamma}(P_{8})$ to ${\mathbb Z}$ defined by (\ref{c}) with respect to the $\triangle Y$-exchange at $\triangle_{2}$. In the following we show $\tilde{\omega}_{Q_{7}}=\omega_{P_{8}}$. Let $\gamma$ be an element in $\Gamma(P_{8})$. 
If $\gamma$ belongs to $\Gamma_{8}(P_{8})$, then there uniquely exists an element $\gamma'$ in $\Gamma_{7}(Q_{7})$ such that $\bar{\Phi}^{-1}_{Q_{7},P_{8}}(\gamma)=\{\gamma'\}$. Thus we have $\tilde{\omega}_{Q_{7}}(\gamma)=\omega_{Q_{7}}(\gamma')=1$. 
If $\gamma$ belongs to $\Gamma_{7}(P_{8})$, then $\gamma$ does not contain exactly one of the vertices of $P_{8}$. Then it is sufficient to consider the following four cases up to symmetry of $P_{8}$: (i) If $\gamma$ does not contain one of $x$ and $y$, then we may assume that $\gamma$ contains $x$. Then there uniquely exists an element $\gamma'$ in $\Gamma_{7}(Q_{7})$ such that $\bar{\Phi}^{-1}_{Q_{7},P_{8}}(\gamma)=\{\gamma'\}$. Thus we have $\tilde{\omega}_{Q_{7}}(\gamma)=\omega_{Q_{7}}(\gamma')=1$. 
(ii) If $\gamma$ does not contain $w$, then there exists an element $\gamma'_{1}$ in $\Gamma_{6}(Q_{7})$ and an element $\gamma'_{2}$ in $\Gamma_{7}(Q_{7})$ such that $\bar{\Phi}^{-1}_{Q_{7},P_{8}}(\gamma)=\{\gamma'_{1},\gamma'_{2}\}$. Since $\gamma'_{1}$ also does not contain $w$, we have $\tilde{\omega}_{Q_{7}}(\gamma)=\omega_{Q_{7}}(\gamma'_{1})+\omega_{Q_{7}}(\gamma'_{2})=0+1=1$. 
(iii) If $\gamma$ does not contain one of $u,v,u'$ and $v'$ (in other words, $\gamma$ contains all of $x,y,w$ and $A$), then we may assume that $\gamma$ does not contain $u$. Then there uniquely exists an element $\gamma'$ in $\Gamma_{6}(Q_{7})$ such that $\bar{\Phi}^{-1}_{Q_{7},P_{8}}(\gamma)=\{\gamma'\}$. Since $\gamma'$ also does not contain $u$, we have $\tilde{\omega}_{Q_{7}}(\gamma)=\omega_{Q_{7}}(\gamma')=0$.
(iv) If $\gamma$ does not contain $A$, then there uniquely exists an element $\gamma'$ in $\Gamma_{6}(Q_{7})$ such that $\bar{\Phi}^{-1}_{Q_{7},P_{8}}(\gamma)=\{\gamma'\}$.  Since $\gamma'$ contains all of $u,v,w$ and $x$, we have $\tilde{\omega}_{Q_{7}}(\gamma)=\omega_{Q_{7}}(\gamma')=-1$.
If $\gamma$ belongs to $\Gamma_{6}(P_{8})$, then $\gamma$ does not contain exactly two of the vertices of $P_{8}$. Then it is sufficient to consider the following four cases up to symmetry of $P_{8}$: (i) If $\gamma$ contains $w$ and does not contain one of $x$ and $y$, then we may assume that $\gamma$ does not contain $y$. Then there uniquely exists an element $\gamma'$ in $\Gamma_{6}(Q_{7})$ such that $\bar{\Phi}^{-1}_{Q_{7},P_{8}}(\gamma)=\{\gamma'\}$.  Since $\gamma'$ contains all of $u,v,w$ and $x$, we have $\tilde{\omega}_{Q_{7}}(\gamma)=\omega_{Q_{7}}(\gamma')=-1$.
(ii) If $\gamma$ does not contain $w$ and one of $x$ and $y$, then we may assume that $\gamma$ does not contain $y$. Then there uniquely exists an element $\gamma'$ in $\Gamma_{6}(Q_{7})$ such that $\bar{\Phi}^{-1}_{Q_{7},P_{8}}(\gamma)=\{\gamma'\}$.  Since $\gamma'$ does not contain $w$, we have $\tilde{\omega}_{Q_{7}}(\gamma)=\omega_{Q_{7}}(\gamma')=0$.
(iii) If $\gamma$ does not contain either $w$ or $A$, then there exists an element $\gamma'_{1}$ in $\Gamma_{5}(Q_{7})$ and an element $\gamma'_{2}$ in $\Gamma_{6}(Q_{7})$ such that $\bar{\Phi}^{-1}_{Q_{7},P_{8}}(\gamma)=\{\gamma'_{1},\gamma'_{2}\}$. Since $\gamma'_{2}$ contains all of $u,v,w$ and $x$, we have $\tilde{\omega}_{Q_{7}}(\gamma)=\omega_{Q_{7}}(\gamma'_{1})+\omega_{Q_{7}}(\gamma'_{2})=-1-1=-2$. 
(iv) If $\gamma$ contains all of $x,y$ and $w$, then we may assume that $\gamma$ does not contain either $u$ or $u'$. Then there exists an element $\gamma'_{1}$ in $\Gamma_{5}(Q_{7})$ and an element $\gamma'_{2}$ in $\Gamma_{6}(Q_{7})$ such that $\bar{\Phi}^{-1}_{Q_{7},P_{8}}(\gamma)=\{\gamma'_{1},\gamma'_{2}\}$. Since $\gamma'_{2}$ does not contain $u$, we have $\tilde{\omega}_{Q_{7}}(\gamma)=\omega_{Q_{7}}(\gamma'_{1})+\omega_{Q_{7}}(\gamma'_{2})=-1+0=-1$. 
If $\gamma$ belongs to $\Gamma_{5}(P_{8})$, then we have $\tilde{\omega}_{Q_{7}}(\gamma)=-1$ in the same way as the case of $Q_{7}$. 
If $\gamma$ belongs to $\Gamma(P_{8})\setminus \cup_{k=5}^{8}\Gamma_{k}(P_{8})$, we have $\tilde{\omega}_{Q_{7}}(\gamma)=0$. 
In conclusion, we see that 
\begin{eqnarray*}
\tilde{\omega}_{Q_{7}}(\gamma)=\left\{
       \begin{array}{@{\,}ll}
       1 & \mbox{if $\gamma\in \Gamma_{8}(P_{8})\cup \left\{\delta\in \Gamma_{7}(P_{8})\ |\ \delta\not\supset \{x,y,w\}\right\}$}\\
       -1 & \mbox{if $\gamma\in \left\{\delta\in \Gamma_{7}(P_{8})\ |\ \delta\not\ni A\right\}\cup \left\{\delta\in \Gamma_{6}(P_{8})\ |\ \delta\ni w\right\}\cup \Gamma_{5}(P_{8})$}\\
       -2 & \mbox{if $\gamma\in \left\{\delta\in \Gamma_{6}(P_{8})\ |\ \delta\cap \{A,w\}=\emptyset\right\}$}\\
       0 & \mbox{otherwise}
       \end{array}
     \right.
\end{eqnarray*}
for an element $\gamma$ in $\Gamma(P_{8})$. Thus it follows that $\tilde{\omega}_{Q_{7}}=\omega_{P_{8}}$. 

\begin{figure}[htbp]
      \begin{center}
\scalebox{0.4}{\includegraphics*{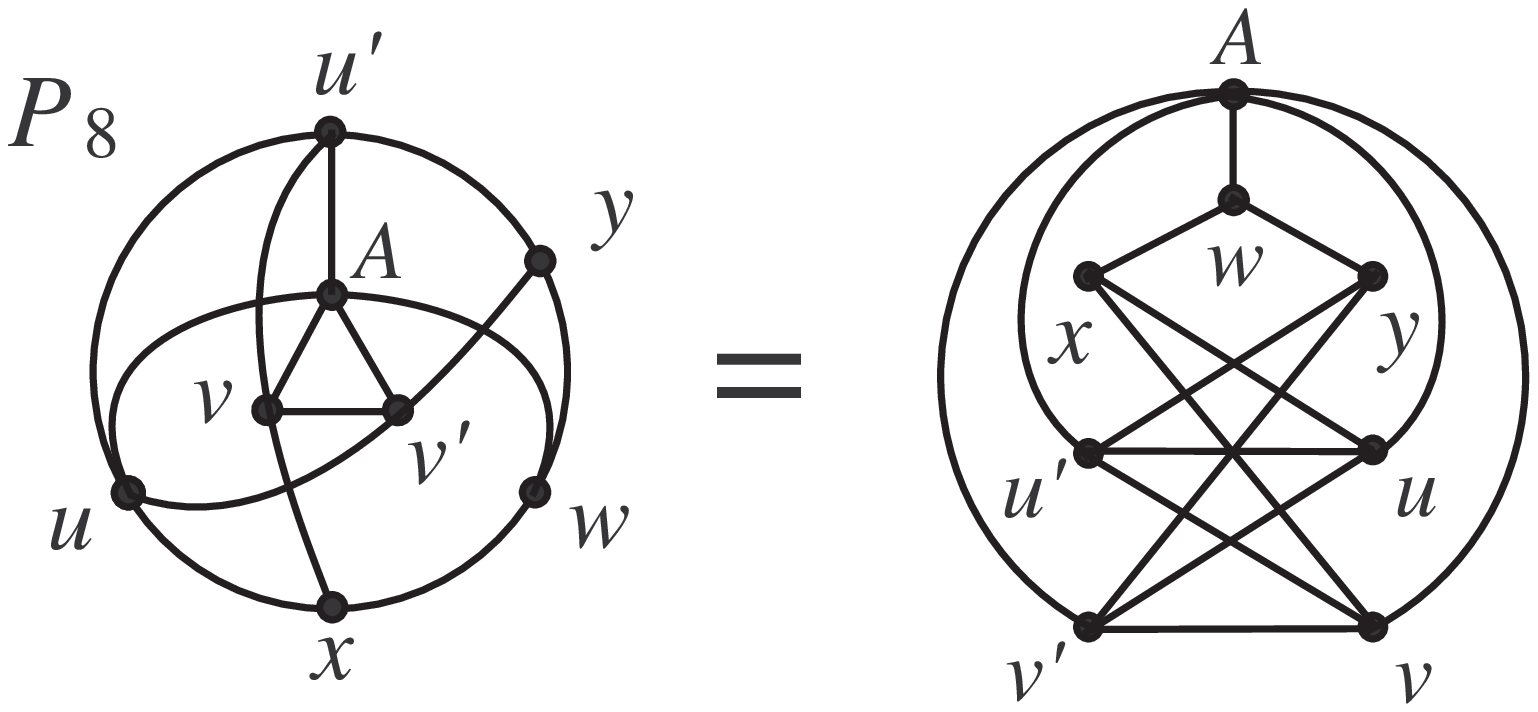}}
      \end{center}
   \caption{}
  \label{P8iso}
\end{figure} 

(6) Let $G$ be $P_{9}$ which is obtained from $P_{8}$ by a single $\triangle Y$-exchange at $\triangle_{3}$. Note that $\{\triangle_{1},\triangle_{2},\triangle_{3}\}$ is stable. Let $\omega_{P_{8}}$ be the map from ${\Gamma}(P_{8})$ to ${\mathbb Z}$ as in (5) and $\tilde{\omega}_{P_{8}}$ the map from ${\Gamma}(P_{9})$ to ${\mathbb Z}$ defined by (\ref{c}) with respect to the $\triangle Y$-exchange at $\triangle_{3}$. In the following we show $\tilde{\omega}_{P_{8}}=\omega_{P_{9}}$. Let $\gamma$ be an element in $\Gamma(P_{9})$. 
If $\gamma$ belongs to $\Gamma_{9}(P_{9})$, then there uniquely exists an element $\gamma'$ in $\Gamma_{8}(P_{8})$ such that $\bar{\Phi}^{-1}_{P_{8},P_{9}}(\gamma)=\{\gamma'\}$. Thus we have $\tilde{\omega}_{P_{8}}(\gamma)=\omega_{P_{8}}(\gamma')=1$. 
If $\gamma$ belongs to $\Gamma_{8}(P_{9})$, then $\gamma$ does not contain exactly one of the vertices of $P_{9}$. Then it is sufficient to consider the following two cases up to symmetry of $P_{9}$: (i) If $\gamma$ contains all of $v,v'$ and $A$, then we may assume that $\gamma$ does not contain one of $x$ and $w$. In any case, there uniquely exists an element $\gamma'$ in $\Gamma_{7}(P_{8})$ such that $\bar{\Phi}^{-1}_{P_{8},P_{9}}(\gamma)=\{\gamma'\}$. Since $\gamma'$ also does not contain one of $x$ and $w$, we have $\tilde{\omega}_{P_{8}}(\gamma)=\omega_{P_{8}}(\gamma')=1$. (ii) If $\gamma$ does not contain one of $v,v'$ and $A$, then we may assume that $\gamma$ does not contain $v'$. Then there uniquely exists an element $\gamma'$ in $\Gamma_{7}(P_{8})$ such that $\bar{\Phi}^{-1}_{P_{8},P_{9}}(\gamma)=\{\gamma'\}$. Since $\gamma'$ contains all of $x,y$ and $w$, we have $\tilde{\omega}_{P_{8}}(\gamma)=\omega_{P_{8}}(\gamma')=0$. 
If $\gamma$ belongs to $\Gamma_{7}(P_{9})$, then $\gamma$ does not contain exactly two of the vertices of $P_{9}$. Then it is sufficient to consider the following two cases up to symmetry of $P_{9}$: (i) If $\gamma$ contains all of $v,v'$ and $A$, then we may assume that $\gamma$ does not contain either $x$ or $w$. Then there uniquely exists an element $\gamma'$ in $\Gamma_{6}(P_{8})$ such that $\bar{\Phi}^{-1}_{P_{8},P_{9}}(\gamma)=\{\gamma'\}$. Since $\gamma'$ contains $A$ and does not contain $w$, we have $\tilde{\omega}_{P_{8}}(\gamma)=\omega_{P_{8}}(\gamma')=0$. 
(ii) If $\gamma$ does not contain one of $v,v'$ and $A$, then we may assume that $\gamma$ does not contain either $x$ or $v$, or $\gamma$ does not contain either $u'$ or $v$. In the former case, there uniquely exists an element $\gamma'$ in $\Gamma_{6}(P_{8})$ such that $\bar{\Phi}^{-1}_{P_{8},P_{9}}(\gamma)=\{\gamma'\}$. Since $\gamma'$ contains $w$, we have $\tilde{\omega}_{P_{8}}(\gamma)=\omega_{P_{8}}(\gamma')=-1$. In the latter case, there exists an element $\gamma'_{1}$ in $\Gamma_{6}(P_{8})$ and an element $\gamma'_{2}$ in $\Gamma_{7}(P_{8})$ such that $\bar{\Phi}^{-1}_{P_{8},P_{9}}(\gamma)=\{\gamma'_{1},\gamma'_{2}\}$. Since $\gamma'_{1}$ contains $w$ and $\gamma'_{2}$ contains all of $x,y,w$ and $A$, we have $\tilde{\omega}_{P_{8}}(\gamma)=\omega_{P_{8}}(\gamma'_{1})+\omega_{P_{8}}(\gamma'_{2})=-1+0=-1$. 
If $\gamma$ belongs to $\Gamma_{6}(P_{9})$, then $\gamma$ does not contain exactly three of the vertices of $P_{9}$. Then it is sufficient to consider the following two cases up to symmetry of $P_{9}$: (i) If $\gamma$ contains all of $v,v'$ and $A$, then we may assume that $\gamma$ does not contain any of $x,y$ or $w$, or $\gamma$ does not contain any of $u,u'$ or $z$. In the former case, there uniquely exists an element $\gamma'$ in $\Gamma_{5}(P_{8})$ such that $\bar{\Phi}^{-1}_{P_{8},P_{9}}(\gamma)=\{\gamma'\}$. Thus we have $\tilde{\omega}_{P_{8}}(\gamma)=\omega_{P_{8}}(\gamma')=-1$. In the latter case, there uniquely exists an element $\gamma'$ in $\Gamma_{6}(P_{8})$ such that $\bar{\Phi}^{-1}_{P_{8},P_{9}}(\gamma)=\{\gamma'\}$. Since $\gamma'$ contains $w$, we have $\tilde{\omega}_{P_{8}}(\gamma)=\omega_{P_{8}}(\gamma')=-1$. 
(ii) If $\gamma$ does not contain one of $v,v'$ and $A$, then we may assume that $\gamma$ does not contain any of $x,u'$ or $v$, or $\gamma$ does not contain any of $v,v'$ or $A$. In any case, there exists an element $\gamma'_{1}$ in $\Gamma_{5}(P_{8})$ and an element $\gamma'_{2}$ in $\Gamma_{6}(P_{8})$ such that $\bar{\Phi}^{-1}_{P_{8},P_{9}}(\gamma)=\{\gamma'_{1},\gamma'_{2}\}$. Since $\gamma'_{2}$ contains $w$, we have $\tilde{\omega}_{P_{8}}(\gamma)=\omega_{P_{8}}(\gamma'_{1})+\omega_{P_{8}}(\gamma'_{2})=-1-1=-2$. 
If $\gamma$ belongs to $\Gamma_{5}(P_{9})$, then we have $\tilde{\omega}_{P_{8}}(\gamma)=-1$ in the same way as the case of $Q_{7}$. 
If $\gamma$ belongs to $\Gamma(P_{9})\setminus \cup_{k=5}^{9}\Gamma_{k}(P_{9})$, we have $\tilde{\omega}_{P_{8}}(\gamma)=0$. In conclusion, we see that 
\begin{eqnarray*}
&&{\omega_{P_{9}}}(\gamma)\\
&=&\left\{
       \begin{array}{@{\,}ll}
       1 & \mbox{if $\gamma\in \Gamma_{9}(P_{9})\cup \left\{\delta\in \Gamma_{8}(P_{9})\ |\ \delta \supset \left\{A,v,v'\right\}\right\}$}\\
       -1 & \mbox{if $\gamma\in \left\{\delta\in \Gamma_{7}(P_{9})\ |\ \delta\not\supset \left\{A,v,v'\right\}\right\}\cup \left\{\delta\in \Gamma_{6}(P_{9})\ |\ \delta \supset \left\{A,v,v'\right\}\right\}\cup \Gamma_{5}(P_{9})$}\\
       -2 & \mbox{if $\gamma\in \left\{\delta\in \Gamma_{6}(P_{9})\ |\ \delta\not\supset \left\{A,v,v'\right\}\right\}$}\\
       0 & \mbox{otherwise}
       \end{array}
     \right.
\end{eqnarray*}
for an element $\gamma$ in $\Gamma(P_{9})$. Thus it follows that $\tilde{\omega}_{P_{8}}=\omega_{P_{9}}$. 

(7) Let $G$ be $P_{10}$ which is obtained from $P_{9}$ by a single $\triangle Y$-exchange at $\triangle_{4}$. Note that $\{\triangle_{1},\triangle_{2},\triangle_{3},\triangle_{4}\}$ is stable. Let $\omega_{P_{9}}$ be the map from ${\Gamma}(P_{9})$ to ${\mathbb Z}$ as in (6) and $\tilde{\omega}_{P_{9}}$ the map from ${\Gamma}(P_{10})$ to ${\mathbb Z}$ defined by (\ref{c}) with respect to the $\triangle Y$-exchange at $\triangle_{4}$. In the following we show $\tilde{\omega}_{P_{9}}=\omega_{P_{10}}$. Let $\gamma$ be an element in $\Gamma(P_{10})$. Note that $\Gamma_{k}(P_{10})$ is the empty set if $k\neq 5,6,8,9$. 
If $\gamma$ belongs to $\Gamma_{9}(P_{10})$, then $\gamma$ does not contain exactly one of the vertices of $P_{10}$. Then it is sufficient to consider the following two cases up to symmetry of $P_{10}$: (i) If $\gamma$ does not contain $B$, then there uniquely exists an element $\gamma'$ in $\Gamma_{9}(P_{9})$ such that $\bar{\Phi}^{-1}_{P_{9},P_{10}}(\gamma)=\{\gamma'\}$. Thus we have $\tilde{\omega}_{P_{9}}(\gamma)=\omega_{P_{9}}(\gamma')=1$. (ii) If $\gamma$ contains $B$, then we may assume that $\gamma$ does not contain $A$. Then there exists an element $\gamma'_{1}$ in $\Gamma_{8}(P_{9})$ and an element $\gamma'_{2}$ in $\Gamma_{9}(P_{9})$ such that $\bar{\Phi}^{-1}_{P_{9},P_{10}}(\gamma)=\{\gamma'_{1},\gamma'_{2}\}$. Since $\gamma'_{1}$ does not contain $A$, we have $\tilde{\omega}_{P_{9}}(\gamma)=\omega_{P_{9}}(\gamma'_{1})+\omega_{P_{9}}(\gamma'_{2})=0+1=1$. 
If $\gamma$ belongs to $\Gamma_{8}(P_{10})$, then $\gamma$ does not contain exactly two of the vertices of $P_{10}$. Then we may assume that $\gamma$ does not contain either $x$ or $w$, or $\gamma$ does not contain either $A$ or $w$. In the former case, there uniquely exists an element $\gamma'$ in $\Gamma_{7}(P_{9})$ such that $\bar{\Phi}^{-1}_{P_{9},P_{10}}(\gamma)=\{\gamma'\}$. Since $\gamma'$ contains all of $A,v$ and $v'$, we have $\tilde{\omega}_{P_{9}}(\gamma)=\omega_{P_{9}}(\gamma')=0$. In the latter case, there exists an element $\gamma'_{1}$ in $\Gamma_{7}(P_{9})$ and an element $\gamma'_{2}$ in $\Gamma_{8}(P_{9})$ such that $\bar{\Phi}^{-1}_{P_{9},P_{10}}(\gamma)=\{\gamma'_{1},\gamma'_{2}\}$. Since $\gamma'_{1}$ does not contain $A$ and $\gamma'_{2}$ contains all of $A,v$ and $v'$, we have $\tilde{\omega}_{P_{9}}(\gamma)=\omega_{P_{9}}(\gamma'_{1})+\omega_{P_{9}}(\gamma'_{2})=-1+1=0$. 
If $\gamma$ belongs to $\Gamma_{6}(P_{10})$, then $\gamma$ does not contain exactly four of the vertices of $P_{10}$. Then we may assume that $\gamma$ does not contain any of $A,B,v$ or $v'$ or $\gamma$ does not contain any of $A,B,z$ or $w$. In any case, there uniquely exists an element $\gamma'$ in $\Gamma_{6}(P_{9})$ such that $\bar{\Phi}^{-1}_{P_{9},P_{10}}(\gamma)=\{\gamma'\}$. Since $\gamma'$ does not contain one of $A,v$ and $v'$, we have $\tilde{\omega}_{P_{9}}(\gamma)=\omega_{P_{9}}(\gamma')=-2$. If $\gamma$ belongs to $\Gamma_{5}(P_{10})$, then we have $\tilde{\omega}_{P_{8}}(\gamma)=-1$ in the same way as the case of $Q_{7}$. In conclusion, we see that 
\begin{eqnarray*}
\tilde{\omega}_{P_{10}}(\gamma)=\left\{
       \begin{array}{@{\,}ll}
       1 & \mbox{if $\gamma\in \Gamma_{9}(P_{10})$}\\
       -1 & \mbox{if $\gamma\in \Gamma_{5}(P_{10})$}\\
       -2 & \mbox{if $\gamma\in \Gamma_{6}(P_{10})$}\\
       0 & \mbox{otherwise}
       \end{array}
     \right.
\end{eqnarray*}
for an element $\gamma$ in $\Gamma(P_{10})$. Thus it follows that $\tilde{\omega}_{P_{9}}=\omega_{P_{10}}$. This completes the proof. 
\end{proof}

\begin{Remark}
Let us denote the $3$-cycle $[Axy]$ of $P_{7}$ by $\triangle_{6}$. Note that $P_{8}$ is obtained from $P_{7}$ by a single $\triangle Y$-exchange at $\triangle_{6}$. Let $\omega_{P_{7}}$ be the map from ${\Gamma}(P_{7})$ to ${\mathbb Z}$ as in Theorem \ref{petersen_refine} (4) and $\tilde{\omega}_{P_{7}}$ the map from ${\Gamma}(P_{8})$ to ${\mathbb Z}$ defined by (\ref{c}) with respect to the $\triangle Y$-exchange at $\triangle_{6}$. Then it can be shown that $\tilde{\omega}_{P_{7}}$ coincides with ${\omega}_{P_{8}}$. 
\end{Remark}


%
{\normalsize
}

\end{document}